\documentclass{article}
\usepackage[utf8]{inputenc}
\usepackage{amsthm}
\usepackage{amsmath}
\usepackage{amsfonts}
\usepackage{color}
\usepackage{latexsym}
\usepackage{geometry}
\usepackage{enumerate}
\usepackage[shortlabels]{enumitem}
\usepackage{csquotes}
\usepackage{graphicx}
\usepackage{comment}
\usepackage{appendix}
\usepackage{hyperref}
\usepackage{tikz-cd} 
\hypersetup{
    colorlinks=true,
    linkcolor=blue,
    filecolor=magenta,      
    urlcolor=cyan,
}
\usepackage{bm}
\usepackage{amssymb}
\usepackage{tikz-cd}

\emergencystretch=\maxdimen
\DeclareMathOperator{\Rm}{Rm}

\makeatletter
\newcommand*{\rom}[1]{\rm {\expandafter\@slowromancap\romannumeral #1@}}
\makeatother

\def\XXint#1#2#3{{\setbox0=\hbox{$#1{#2#3}{\int}$ }
\vcenter{\hbox{$#2#3$ }}\kern-.6\wd0}}

\protected\def\vts{%
  \ifmmode
    \mskip0.5\thinmuskip
  \else
    \ifhmode
      \kern0.08334em
    \fi
  \fi
}

\numberwithin{equation}{section}
\newtheorem{Theorem}{Theorem}[section]
\newtheorem{Proposition}[Theorem]{Proposition}
\newtheorem{Lemma}[Theorem]{Lemma}
\newtheorem{Corollary}[Theorem]{Corollary}

\theoremstyle{definition}
\newtheorem{Definition}[Theorem]{Definition}

\title{Local smooth convergence of $\mathbb{F}$-limit flows}
\author{Pak-Yeung Chan\footnote{Pak-Yeung Chan's research is supported by EPSRC grant EP/T019824/1.}, Zilu Ma, Yongjia Zhang\footnote{Yongjia Zhang's research is partially supported by Shanghai Sailing Program 23YF1420400 and Research Start-up Fund of SJTU WH220407110.}}
\date{}

\begin{document}


\maketitle
\begin{abstract}
The metric flow is introduced and extensively studied by Bamler \cite{Bam20b, Bam20c}, especially as an $\mathbb{F}$-limit of a sequence of smooth Ricci flows with uniformly bounded Nash entropy, in which case each regular point on the limit is a point of smooth convergence. In this note, we shall consider the $\mathbb{F}$-convergence of a sequence of $\mathbb{F}$-limit flows, and, like Bamler, show that each regular point on the limit is also a point of smooth convergence. The main result will be applied in a forthcoming work of the authors \cite{CMZ23a}.
\end{abstract}

\bigskip 

\section{Introduction}



The metric flow \cite{Bam20b,Bam20c} is a reasonable generalization of the Ricci flow. A metric flow that arises as an $\mathbb{F}$-limit of smooth Ricci flows with uniformly bounded Nash entropy is in particular worthy of detailed study, since it is a generalization of the classical singularity model, and is essential to extending the Ricci flow beyond singularities. In view of the fact that such  metric flow provides an important category for the Ricci flow singularity analysis, it shall be called the \emph{ $\mathbb{F}$-limit flow} by us, whose definition will be presented shortly.

Because of the massive content of \cite{Bam20a,Bam20b, Bam20c}, and of the lengthy definitions of the notions encountered in this note, such as, metric flows, $\mathbb{F}$-convergence, and convergence within a correspondence, we shall not introduce any of them in details, but assume  the reader's  familiarity with the definitions and notations in Bamler's three papers. All notations which we adopt are the same as Bamler's, and all definitions mentioned can be found in \cite{Bam20a,Bam20b, Bam20c}.

Throughout this note, we assume that $I$ is a left-open right-closed interval of the form $(-T,S]$, where $T\le +\infty$, $S<+\infty$. We shall use the capital letter $\mathcal{M}$ (sometimes with indices) to represent a Ricci flow, namely, a space-time produce $M\times I$ with a Riemannian metric evolving according to the Ricci flow equation. The lowercase letters $x$, $y$, etc., always represent space-time points in a Ricci flow or points in a metric flow. The letter $\mathfrak{t}$ stands for the time function. For instance, if $x=(a_0,t_0)\in\mathcal{M}=M\times I$, then $\mathfrak{t}(x)=t_0$. As \cite{Bam20b}, the notations $\mathcal{X}_t$ and $\mathcal{X}_{I_0}$ represent a time-slice and a time-slab of the metric flow $\mathcal{X}$, respectively; the same notations are also applied to Ricci flows. All smooth Ricci flows under our consideration are complete and have bounded curvature on compact time-intervals.

\begin{Definition}[$\mathbb{F}$-limit flow]\label{F-limit-flow}
    A metric flow $\big(\mathcal{X},(\mu_t)_{t\in I'}\big)$ over interval $I$, where $I'\subset I$ and $|I\setminus I'|=0$, is called an $n$-dimensional \emph{$\mathbb{F}$-limit flow}, if there is a sequence of Ricci flows $\{\mathcal{M}^i\}_{i=1}^\infty$ over $I$ and a sequence of points $x^i_0\in\mathcal{M}^i_{\sup I}$, such that
    \begin{gather}\label{limiting-definition}
        \big(\mathcal{M}^i,(\nu^i_{x^i_0;t})_{t\in I}\big)\xrightarrow[i\to\infty]{\quad\mathbb{F}\quad}\big(\mathcal{X},(\mu_t)_{t\in I'}\big).
    \end{gather}
    If $I$ is unbounded, then the convergence is assumed to be on compact time intervals.

    If, in addition to \eqref{limiting-definition}, there are positive numbers $\tau>0$ and $Y<+\infty$, such that
    \begin{gather}
        \mathcal{N}_{x_0^i}(\tau)\ge-Y,
    \end{gather}
    then $\big(\mathcal{X},(\mu_t)_{t\in I'}\big)$ is called \emph{$(\tau, Y)$-noncollapsed}.
\end{Definition}

\textbf{Remark. }\emph{It is to be admitted that our definition of $\mathbb{F}$-limit flow is a bit too restrictive. Indeed, in a more reasonably general definition, each Ricci flow $\mathcal{M}^i$ in \eqref{limiting-definition} should live in its own time-interval $I^i=(-T^i,S]$ with $T^i\to T$, or, even more generally, $I^i=(-T^i,S^i]$ with $T^i\to T$ and $S^i\to S$. However, since smooth convergence point is a local notion, one may always restrict all flows to a common interval, and Theorem \ref{main-theorem} still holds with the more general definition. In particular, we remark that this does not affect our application in \cite{CMZ23a}.}
\\

According to \cite[\S 7]{Bam20b}, the space of $n$-dimensional $\mathbb{F}$-limit flows is also sequentially compact with respect to the $\mathbb{F}$-topology. We shall consider a sequence of $\mathbb{F}$-limit flows, which converges in the $\mathbb{F}$-sense to a metric flow. The limit is naturally an $\mathbb{F}$-limit flow. Our main result is that, if the sequence consists of $(\tau,Y)$-noncollapsed $\mathbb{F}$-limit flows, then the convergence is smooth in the sense of \cite[Theorem 9.31]{Bam20b} wherever the limit is smooth.

\begin{Theorem}[Main Theorem]\label{main-theorem}
    Let $\big(\mathcal{X},(\mu_t)_{t\in I'}\big)$ be an $n$-dimensional $\mathbb{F}$-limit flow over interval $I$. Assume that there is a sequence of $n$-dimensional $(\tau,Y)$-noncollapsed $\mathbb{F}$-limit metric flows $\left\{\big(\mathcal{X}^i,(\mu^i_t)_{t\in I^{',i}}\big)\right\}_{i=1}^\infty$ over $I$, where $\tau>0$ and $Y<+\infty$ are some constants independent of $i$, such that
    \begin{align}\label{the-convergence-in-question}
        \big(\mathcal{X}^i,(\mu^i_t)_{t\in I^{',i}}\big)\xrightarrow[i\to\infty]{\quad\mathbb{F}\quad}\big(\mathcal{X},(\mu_t)_{t\in I'}\big).
    \end{align}
    Then the above convergence is smooth in the sense of \cite[Theorem 9.31]{Bam20b} on $\mathcal{R}_{(-\infty,\sup I)\cap I}\subset\mathcal{X}$. In other words, for each $y\in\mathcal{R}_{(-\infty,\sup I)\cap I}\subset\mathcal{X}$, we have $y\in\mathcal{R}^*\subset\mathcal{X}$.
\end{Theorem}

 \textbf{Remark.} \emph{In our setting, there are always multiple convergences considered. For instance, the defining convergence \eqref{limiting-definition} and the convergence \eqref{the-convergence-in-question} studied in the main theorem may not have the same set of smooth convergence. Whenever necessary, we shall use subindices to indicate the convergence in question. For example, $\mathcal{R}^*_{\eqref{the-convergence-in-question}}\subset\mathcal{X}$ represents the points of smooth convergence in \eqref{the-convergence-in-question}.}

\begin{Corollary}
    Let $\mathcal{M}$ be an $n$-dimensional Ricci flow over interval $I$. $(\mu_t)_{t\in I}$ is a conjugate heat flow on $\mathcal{M}$. Assume that $\{(\mathcal{X}^i,(\mu_t)_{t\in I^{',i}})\}_{i=1}^\infty$ is a sequence of $n$-dimensional $(\tau,Y)$-noncollapsed $\mathbb{F}$-limit flows over $I$ converging to $(\mathcal{M},(\mu_t)_{t\in I})$ in the $\mathbb{F}$-sense, then the convergence is locally smooth on $\mathcal{M}_{(-\infty,\sup I)\cap I}$. Precisely, there are open sets $U_i\subset \mathcal{M}$ satisfying $U_1\subset U_2\subset U_3\subset\hdots$ and $\cup_{i=1}^\infty U_i= \mathcal{M}_{(-\infty,\sup I)\cap I}$, time-preserving diffeomorphisms 
    \begin{align*}
        \psi^i: U_i\to V_i\subset\mathcal{R}^i\subset\mathcal{X}^i,
    \end{align*}
    and a sequence of positive numbers $\varepsilon_i\searrow 0$, such that all properties in \cite[Theorem 9.31(a)---(f)]{Bam20b} hold. In other words, each $\psi^i$ is $\varepsilon_i$-close to an isometry.
    
\end{Corollary}

\section{Regular points on $\mathbb{F}$-limit flows}

The following theorem is a generalization of \cite[Theorem 10.4]{Bam20a}.  It says that if the metric flow is regular enough at a point, then the Nash entropy cannot be too far away from zero. Recall that the Nash entropy is defined for a $(\tau, Y)$-noncollapsed $\mathbb{F}$-limit flow since the singular set has zero measure. The Nash entropy is also continuous with respect to $\mathbb{F}$-convergence according to \cite[Theorem 2.10]{Bam20c}.

\begin{Theorem}\label{theorem-nash-control}
    For any $\varepsilon>0$ and any $Y<+\infty$, there is a positive number $\delta(Y,\varepsilon)>0$, depending also on the dimension $n$, with the following property. Let $\big(\mathcal{X},(\mu_t)_{t\in I'}\big)$ be an $n$-dimensional $(\tau_0,Y_0)$-noncollapsed  $\mathbb{F}$-limit flow for some $\tau_0>0$ and $Y_0<+\infty$. For any $x\in\mathcal{X}_{(-\infty,\sup I)\cap I}$ and any $r>0$, if $[\mathfrak{t}(x)-r^2,\mathfrak{t}(x)]\subset I$, $P^-(x;r)\Subset\mathcal{R}$ is unscathed, and
    \begin{gather*}
        |{\Rm}|\le r^{-2}\quad\text{on}\quad P^-(x;r),
        \\
        \mathcal{N}_x(r^2)\ge-Y,
    \end{gather*}
    then we have 
    $$\mathcal{N}_x(\delta r^2)\ge -\varepsilon.$$
\end{Theorem}

\begin{proof}
    This theorem follows from a limiting argument and the continuity of the Nash entropy with respect to the $\mathbb{F}$-convergence \cite[Theorem 2.10]{Bam20c}.  Let us assume that the interval $I$ has finite length. The more general case follows easily from our argument, and is left to the reader.

    For any $\varepsilon>0$, $Y<+\infty$, we pick $\delta=\frac{1}{4}\delta_0(\varepsilon,2Y)>0$, where $\delta_0(\cdot,\cdot)$ is the constant in the statement of \cite[Theorem 10.4]{Bam20a}.
     Let $\mathcal{M}^i$ be a sequence of $n$-dimensional Ricci flows and $x_0^i\in\mathcal{M}^i_{\sup I}$ a sequence of points as in \eqref{limiting-definition}, such that
    \begin{gather*}
        \mathcal{N}_{x^i_0}(\tau_0)\ge-Y_0,
        \\
        \big(\mathcal{M}^i,(\nu_{x^i_0;t})_{t\in I}\big)\xrightarrow[i\to\infty]{\quad \mathbb{F}\quad }\big(\mathcal{X},(\mu_t)_{t\in I'}\big).
    \end{gather*}

    Since $\mathbb{F}$-convergence is equivalent to convergence within a correspondence (c.f. \cite[\S 5]{Bam20b}), let us pick a correspondence $\mathfrak{C}=(Z_t,\varphi^i_t)$ over $I$ between $\big(\mathcal{M}^i,(\nu_{x^i_0;t})_{t\in I}\big)$ and $\big(\mathcal{X},(\mu_t)_{t\in I'}\big)$, such that
\begin{align}\label{extranonsense10}
    \big(\mathcal{M}^i,(\nu_{x^i_0;t})_{t\in I}\big)\xrightarrow[i\to\infty]{\quad \mathbb{F},\ \mathfrak{C}\quad }\big(\mathcal{X},(\mu_t)_{t\in I'}\big).
\end{align}
By \cite[Theorem 6.45]{Bam20b}, there are points $x^i\in\mathcal{M}^i_{\mathfrak{t}(x)}$, such that
\begin{align*}
    x^i\xrightarrow[i\to\infty]{\quad\mathfrak{C}\quad} x.
\end{align*}

On the other hand, since \cite[Theorem 2.5]{Bam20c} implies that $P^-(x;r)\Subset\mathcal{R}=\mathcal{R}^*_{\eqref{extranonsense10}}$, we can apply \cite[Theorem 9.31]{Bam20b} to find an open set $P^-(x;r)\Subset U\subset\mathcal{R}^*$ and time-preserving diffeomorphisms
\begin{align*}
    \psi_i: U\rightarrow V^i\subset \mathcal{M}^i
\end{align*}
with all the properties in \cite[Theorem 9.31]{Bam20b}. In particular, we have $\psi_i^{-1}(x^i)\rightarrow x$ and
\begin{align}\label{extranonsense1}
    |{\Rm^i}|\le 4r^{-2}\qquad \text{ on }\qquad P^-(x^i;\tfrac{1}{2}r)
\end{align}
for all $i$ large enough.

By the continuity of the Nash entropy \cite[Theorem 2.10]{Bam20c}, we have that $\displaystyle\lim_{i\to\infty}\mathcal{N}_{x^i}(r^2)=\mathcal{N}_x(r^2)\ge -Y$. Thus
\begin{align}\label{extranonsense2}
    \mathcal{N}_{x^i}\big(\tfrac{1}{4}r^2\big)\ge -2Y
\end{align}
for all $i$ large enough.

Applying \cite[Theorem 10.4]{Bam20a} at each $x^i$ using \eqref{extranonsense1} \eqref{extranonsense2}, we have
\begin{align*}
    \mathcal{N}_{x^i}(\tfrac{1}{4}\delta_0r^2)\ge -\varepsilon
\end{align*}
for all $i$ large enough, where $\delta_0=\delta_0(\varepsilon,2Y)$ is the constant in \cite[Theorem 10.4]{Bam20a}. Finally, from the continuity of the Nash entropy \cite[Theorem 2.10]{Bam20c} again and our choice of $\delta$, the conclusion follows.

\end{proof}

\section{Proof of the main theorem}

\subsection{Setting up}

For the sake of simplicity, we assume that the left-open interval $I\subset \mathbb{R}$ is of finite length. The other case  can be easily carried out similarly, since regular point and point of smooth convergence are both local notions.

Let $\big\{(\mathcal{X}^i,(\mu^i_t)_{t\in I^{',i}})\big\}_{i=1}^\infty$ and $(\mathcal{X},(\mu_t)_{t\in I'})$ be the flow pairs over $I$ as in the statement of Theorem \ref{main-theorem}. By Definition \ref{F-limit-flow}, we can find $n$-dimensional smooth Ricci flows $\mathcal{M}^i_j$ over $I$, each with complete time-slices and bounded curvature on compact time-intervals, and points $x^i_j\in \mathcal{M}^i_{j,\sup I}$, such that for each $i$ 
\begin{gather}
    \big(\mathcal{M}^i_j,(\nu_{x^i_j;t})_{t\in I}\big)\xrightarrow[j\to\infty]{\quad\mathbb{F}\quad} (\mathcal{X}^i,(\mu^i_t)_{t\in I^{',i}}),
    \\
    \mathcal{N}_{x^i_j}(\tau)>-Y. \label{noncollapsing-assumption}
\end{gather}

Since $\mathbb{F}$-convergence is equivalent to convergence within a correspondence, we can find a correspondence $\mathfrak{C}=(Z_t,\varphi^i_t)_{t\in I}$ over $I$ between $(\mathcal{X}^i,(\mu^i_t)_{t\in I^{',i}})$ ($i\in\mathbb{N}$) and $(\mathcal{X},(\mu_t)_{t\in I'})$, and for each $i$ a correspondence $\mathfrak{C}^i=(Z^i_t,\varphi^i_{j,t})_{t\in I}$ over $I$ between $\big(\mathcal{M}^i_j,(\nu_{x^i_j;t})_{t\in I}\big)$ ($j\in\mathbb{N}$) and $(\mathcal{X}^i,(\mu^i_t)_{t\in I^{',i}})$, satisfying
\begin{gather}\label{convergence-1}
\big(\mathcal{M}^i_j,(\nu_{x^i_j;t})_{t\in I}\big)\xrightarrow[j\to\infty]{\quad\mathfrak{C}^i,\ \mathbb{F}\quad} (\mathcal{X}^i,(\mu^i_t)_{t\in I^{',i}})
\\
(\mathcal{X}^i,(\mu^i_t)_{t\in I^{',i}}) \xrightarrow[i\to\infty]{\quad\mathfrak{C},\ \mathbb{F}\quad}   (\mathcal{X},(\mu_t)_{t\in I'}).\label{convergence-2}
\end{gather}
To avoid the notational ambiguity, we let $\varphi^i_{\infty,t}:\mathcal{X}^i_t\to Z^i_t$ be the isometric embedding associated with $\mathfrak{C}^i$, and $\varphi^i_t:\mathcal{X}^i_t\to Z_t,\  \varphi_t^{\infty}:\mathcal{X}_t\to Z_t$  be the ones associated with $\mathfrak{C}$.


We will also pass to subsequence whenever necessary, by resorting to the argument in \cite{Bam20c}, namely, if for each subsequence there is a further subsequence with some convergence property, then the same can be predicated of the original sequence. For instance, if for any subsequence of $\{(\mathcal{X}^i,\mu^i_t)\}$, there is a further subsequence, such that $y\in\mathcal{X}$ is a point of smooth convergence in \eqref{convergence-2}, then $y$ is a point of smooth convergence for the original sequence. Thus, by passing to subsequences, we may assume that all convergences in \eqref{convergence-1} and \eqref{convergence-2} are almost always time-wise on $I$ (c.f. \cite[Theorem 7.6]{Bam20b}).

Next, let $y\in\mathcal{R}_{I\cap(-\infty,\sup I)}\subset\mathcal{X}_{I\cap(-\infty,\sup I)}$ be a point in the \emph{regular part} of $\mathcal{X}$. Let $$t_0:=\mathfrak{t}(y).$$
By the definition of $\mathcal{R}\subset \mathcal{X}$, we can find $r_0>0$, such that $P^-(y;r_0)\Subset\mathcal{R}$ is unscathed and
\begin{gather*}
[t_0-r_0^2,t_0+r_0^2]\subset I, 
\\
|{\Rm}|\le r_0^{-2}\quad\text{ on }\quad P^-(y;r_0).
\end{gather*} 
By the left-openness of $I$, we pick $\inf I<t_1=t_0-r_0^2<t_2<t_0$, and slightly adjust $r_0$ if necessary, such that all the convergences in \eqref{convergence-1} and \eqref{convergence-2} are time-wise at both $t_1$ and $t_2$. According to \cite[Theorem 6.45]{Bam20b} and \cite[Theorem 6.25(b)]{Bam20b}, we can choose points $y^i\in\mathcal{X}^i_{t_0}$ and $y^i_j\in\mathcal{M}^i_{j,t_0}$, such that
\begin{eqnarray}
y^i_j\xrightarrow[j\to\infty]{\quad \mathfrak{C}^i,\ \{t_1,t_2\}\quad} y^i;\qquad y^i\xrightarrow[i\to\infty]{\quad \mathfrak{C},\ \{t_1,t_2\}\quad} y.
\end{eqnarray}

\begin{Lemma}
    There is a constant $Y_1$, such that for each $i$, $\mathcal{N}_{y^i_j}(r_0^2)\ge -Y_1$ whenever $j$ is large enough.
\end{Lemma}

\begin{proof}
By the definition of time-wise $\mathbb{F}$-convergence within a correspondence, we have
\begin{align*}   \lim_{i\to\infty}d^{Z_{t_2}}_{W_1}\big((\varphi^i_{t_2})_*\mu^i_{t_2},(\varphi^\infty_{t_2})_*\mu_{t_2}\big)=0;\qquad \lim_{j\to\infty}d^{Z^i_{t_2}}_{W_1}\big((\varphi^i_{j,t_2})_*\nu_{x^i_j;t_2},(\varphi^i_{\infty,t_2})_*\mu^i_{t_2}\big)=0.
\end{align*}
By the definition of convergence of points within a correspondence \cite[Definition 6.16, Definition 6.18]{Bam20b}, we also have
\begin{align*}
\lim_{i\to\infty}d^{Z_{t_2}}_{W_1}\big((\varphi^i_{t_2})_*\nu_{y^i;t_2},(\varphi^\infty_{t_2})_*\nu_{y;t_2}\big)=0;\qquad \lim_{j\to\infty}d^{Z^i_{t_2}}_{W_1}\big((\varphi^i_{j,t_2})_*\nu_{y^i_j;t_2},(\varphi^i_{\infty,t_2})_*\nu_{y^i;t_2}\big)=0.   
\end{align*}
Thus, we have the following distance estimates
\begin{align*}
\lim_{i\to\infty} d^{\mathcal{X}^i_{t_2}}_{W_1}\big(\nu_{y^i;t_2},\mu^i_{t_2}\big)
    =&\ \lim_{i\to\infty} d^{Z_{t_2}}_{W_1}
\big((\varphi^i_{t_2})_*\nu_{y^i;t_2};(\varphi^i_{t_2})_*\mu^i_{t_2}\big)
    \\
    \le &\ \lim_{i\to\infty}\bigg(d^{Z_{t_2}}_{W_1}\big((\varphi^i_{t_2})_*\mu^i_{t_2},(\varphi^\infty_{t_2})_*\mu_{t_2}\big)+(d^{Z_{t_2}}_{W_1}\big((\varphi^i_{t_2})_*\nu_{y^i;t_2},(\varphi^\infty_{t_2})_*\nu_{y;t_2}\big)\bigg)
    \\
    &\ +d^{Z_{t_2}}_{W_1}\big((\varphi^\infty_{t_2})_*\mu_{t_2},(\varphi^\infty_{t_2})_*\nu_{y;t_2}\big)
    \\
    = &\ d_{W_1}^{\mathcal{X}_{t_2}}\big(\mu_{t_2},\nu_{y;t_2}\big)
    \\
    <&\ \infty
\end{align*} 
and
\begin{align*}
&\lim_{j\to\infty} d^{\mathcal{M}^i_{j,t_2}}_{W_1}\big(\nu_{y^i_j;t_2},\nu_{x^i_j;t_2}\big)\\
    =&\ \lim_{j\to\infty} d^{Z^i_{t_2}}\big((\varphi^i_{j,t_2})_*\nu_{y^i_j;t_2};(\varphi^i_{j,t_2})_*\nu_{x^i_j;t_2}\big)
    \\
    \le &\ \lim_{j\to\infty}\bigg(d^{Z^i_{t_2}}_{W_1}\big((\varphi^i_{j,t_2})_*\nu_{x^i_j;t_2},(\varphi^i_{\infty,t_2})_*\mu^i_{t_2}\big)+(d^{Z^i_{t_2}}_{W_1}\big((\varphi^i_{j,t_2})_*\nu_{y^i_j;t_2},(\varphi^i_{\infty,t_2})_*\nu_{y^i;t_2}\big)\bigg)
    \\
    &\ +d^{Z^i_{t_2}}_{W_1}\big((\varphi^i_{\infty,t_2})_*\mu^i_{t_2},(\varphi^i_{\infty,t_2})_*\nu_{y^i;t_2}\big)
    \\
    = &\ d_{W_1}^{\mathcal{X}^i_{t_2}}\big(\mu^i_{t_2},\nu_{y^i;t_2}\big).
\end{align*}
Thus, by taking away finitely many terms from each sequence $\{\mathcal{M}^i_j\}_{j=1}^\infty$, we may assume that there is a constant $C$ such that
\begin{align}\label{nonsense1}  d^{\mathcal{M}^i_{j,t_2}}_{W_1}\big(\nu_{x^i_j;t_2}, \nu_{y^i_j;t_2}\big)\le C\qquad\text{ for all $i$ and for all $j$}. 
\end{align}

On the other hand, by the noncollapsing assumption \eqref{noncollapsing-assumption}, the standard scalar curvature lower bound, and \cite[Proposition 5.2]{Bam20a}, we have 
\begin{align}\label{nonsense2}
    \mathcal{N}_{x^i_j}(t_0-t_1)\ge -Y'.
\end{align}
The lemma follows from applying \cite[Corollary 5.11]{Bam20a} with conditions \eqref{nonsense1} and \eqref{nonsense2}.
\end{proof}

\subsection{Local regularity along the converging points}

Let $Y_0=\min\{Y_1,\mathcal{N}_y(r_0^2)\}$. Let $\varepsilon=\delta'(\frac{1}{2})$, where $\delta'$ is the constant in \cite[Theorem 10.3]{Bam20a}, and $\delta=\delta(\frac{1}{2}\varepsilon,Y_0)>0$ be the constant in Theorem \ref{theorem-nash-control}. Henceforth, we shall define $$r:=\delta^{\frac{1}{2}}r_0$$ and we have, by Theorem \ref{theorem-nash-control},
\begin{align}\label{nonsense0}
    \mathcal{N}_y(r^2)\ge-\tfrac{1}{2}\varepsilon.
\end{align}

\begin{Proposition}\label{local-regularity}
    For each $i$ large enough, there are infinitely many $j$, such that
    \begin{align*}
        \mathcal{N}_{y^i_j}(r^2)\ge -\varepsilon.
    \end{align*}
\end{Proposition}

\begin{proof}
We argue by contradiction. Suppose the proposition is false, then we can find a sequence $\{i_k\}_{k=1}^\infty$ of counterexamples, such that for each $k$, there are infinitely many $j$ with $\mathcal{N}_{y^{i_k}_j}(r^2)<-\varepsilon$. By our convergence assumptions, and by passing $i_k$ to a subsequence if necessary, we may pick $j_k$ accordingly with the following properties (we write $(\mu_t)_{t\in I'}$ concisely as $\mu_t$ to avoid the notational baggage).
\begin{enumerate}
    \item $\displaystyle \mathcal{N}_{y^{i_k}_{j_k}}(r^2)<-\varepsilon$;
    \item $\displaystyle d^{\mathfrak{C}}_{\mathbb{F}}\big((\mathcal{X}^{i_k},\mu^{i_k}_t),(\mathcal{X},\mu_t)\big)\le \tfrac{1}{k},\qquad  d^{\mathfrak{C}^{i_k}}_{\mathbb{F}}\bigg(\big(\mathcal{M}^{i_k}_{j_k},\nu_{x^{i_k}_{j_k};t}\big),(\mathcal{X}^{i_k},\mu^{i_k}_t)\bigg)\le \tfrac{1}{k}$;
    \item $\displaystyle d^{\mathfrak{C}_{I\cap(-\infty,t_0]}}_{W_1}\bigg((\mathcal{X}_{I\cap(-\infty,t_0]},\nu_{y;t}),(\mathcal{X}^{i_k}_{I\cap(-\infty,t_0]},\nu_{y^{i_k};t})\bigg)\le\frac{1}{k}$,
    \\
    $\displaystyle  d^{\mathfrak{C}^{i_k}_{I\cap(-\infty,t_0]}}_{W_1}\bigg((\mathcal{M}^{i_k}_{j_k,I\cap(-\infty,t_0]},\nu_{y^{i_k}_{j_k};t}),(\mathcal{X}^{i_k}_{I\cap(-\infty,t_0]},\nu_{y^{i_k};t})\bigg)\le\frac{1}{k}$.
\end{enumerate}
The third property above follows from \cite[Theorem 6.40]{Bam20b}.

Next, for each $k$, regarding $\mathfrak{C}^{i_k}$, $\mathfrak{C}$ as correspondences between $\mathcal{M}^{i_k}_{j_k}$ and $\mathcal{X}^{i_k}$, $\mathcal{X}^{i_k}$ and $\mathcal{X}$, respectively, (and omitting many other flows that are also within these correspondences), we combine them into one correspondence $\tilde{\mathfrak{C}}^k$ by applying \cite[Lemma 5.15]{Bam20b}. Note that all the correspondence in our consideration are defined on the same time interval $I$, this avoid many nuances in the definition of the $\mathbb{F}$-distances within a correspondence; one can easily carry this to more general cases. It is easy to verify that
\begin{align*}
    d^{\tilde{\mathfrak{C}}^k}_{\mathbb{F}}\bigg(\big(\mathcal{M}^{i_k}_{j_k},\nu_{x^{i_k}_{j_k};t}\big),(\mathcal{X},\mu_t)\bigg)\le&\  d^{\tilde{\mathfrak{C}}^k}_{\mathbb{F}}\bigg(\big(\mathcal{M}^{i_k}_{j_k},\nu_{x^{i_k}_{j_k};t}\big),(\mathcal{X}^{i_k},\mu^{i_k}_t)\bigg)+d^{\tilde{\mathfrak{C}}^k}_{\mathbb{F}}\bigg((\mathcal{X}^{i_k},\mu^{i_k}_t),(\mathcal{M},\mu_t)\bigg)
    \\
    =&\ d^{\mathfrak{C}^{i_k}}_{\mathbb{F}}\bigg(\big(\mathcal{M}^{i_k}_{j_k},\nu_{x^{i_k}_{j_k};t}\big),(\mathcal{X}^{i_k},\mu^{i_k}_t)\bigg)+d^{\mathfrak{C}}_{\mathbb{F}}\bigg((\mathcal{X}^{i_k},\mu^{i_k}_t),(\mathcal{M},\mu_t)\bigg)
    \\
    \le&\ \frac{2}{k}.
\end{align*}
In like manner, we also have
\begin{align*}
    d^{\tilde{\mathfrak{C}}^k_{I\cap(-\infty,t_0]}}_{W_1}\bigg(\big(\mathcal{M}^{i_k}_{j_k,I\cap(-\infty,t_0]},\nu_{y^{i_k}_{j_k};t}\big),(\mathcal{M}_{I\cap(-\infty,t_0]},\nu_{y;t})\bigg)\le\frac{2}{k}.
\end{align*}
Combining all the $\tilde{\mathfrak{C}}^k$ into one correspondence $\tilde{\mathfrak{C}}$ between $\big(\mathcal{M}^{i_k}_{j_k},\nu_{x^{i_k}_{j_k};t}\big)$ ($k\in\mathbb{N}$) and $(\mathcal{M},\mu_t)$ by way of applying the proof of \cite[Lemma 2.13]{Bam20b} at each $t\in I$, we obtain the following convergence within correspondence $\tilde{\mathfrak{C}}$:
\begin{gather*}   \big(\mathcal{M}^{i_k}_{j_k},\nu_{x^{i_k}_{j_k};t}\big)\xrightarrow[\quad k\to\infty\quad]{\mathbb{F},\ \tilde{\mathfrak{C}}}(\mathcal{X},\mu_t)
\\
y^{i_k}_{j_k}\xrightarrow[\quad k\to\infty\quad]{\tilde{\mathfrak{C}}} y.
\end{gather*}

Finally, by \cite[Theorem 2.10]{Bam20c} and the contradictory assumption $\mathcal{N}_{y^{i_k}_{j_k}}(r^2)<-\varepsilon$, we have that $\mathcal{N}_y(r^2)\le -\varepsilon$. This is a contradiction against \eqref{nonsense0}.
\end{proof}


\subsection{Local conjugate heat flow bound along the converging points}

By Proposition \ref{local-regularity}, for each $i$ we may pass $j$ to a subsequence, such that
\begin{align*}
    \mathcal{N}_{y^i_j}(r^2)\ge-\varepsilon\qquad \text{for all}\qquad j\in\mathbb{N},
\end{align*}
whenever $i$ is large enough. Write $d\mu_t:=u_tdg_t$ on $\mathcal{R}\subset\mathcal{X}$ and $d\nu_{x^i_j;t}:=u^i_{j,t}dg_t$. Define
\begin{align*}
    c_0:=u_{t_0}(y)>0.
\end{align*}
The following proposition follows from a similar argument as  Proposition \ref{local-regularity}.

\begin{Proposition}\label{local-noncollapsing}
    For each $i$ large enough, there are infinitely many $j$, such that
    \begin{align*}
        u^i_{j,t_0}(y^i_j)>\tfrac{1}{2}c_0.
    \end{align*}
\end{Proposition}
\begin{proof}[Sketch of proof] Arguing in the same way as the proof of Proposition \ref{local-regularity}, suppose that we can find a subsequence $\{i_k\}_{k=1}^\infty$, such that for each $i_k$, there are infinitely many $j$ with 
\begin{align*}
    u^{i_k}_{j,t_0}(y^{i_k}_j)\le \tfrac{1}{2}c_0.
\end{align*}
Then we may choose $j_k$ for each $k$ and combine $\mathfrak{C}^{i_k}$ and $\mathfrak{C}$ into a correspondence $\tilde{\mathfrak{C}}$ just like the proof of Proposition \ref{local-regularity}, such that
\begin{gather}   \big(\mathcal{M}^{i_k}_{j_k},\nu_{x^{i_k}_{j_k};t}\big)\xrightarrow[\quad k\to\infty\quad]{\mathbb{F},\ \tilde{\mathfrak{C}}}(\mathcal{X},\mu_t)\label{extranonsense3}
\\
y^{i_k}_{j_k}\xrightarrow[\quad k\to\infty\quad]{\tilde{\mathfrak{C}}} y,\label{extranonsense05}
\\
u^{i_k}_{j_k,t_0}(y^{i_k}_{j_k})\le \tfrac{1}{2}c_0.\label{extranonsense4}
\end{gather}
Since $y\in\mathcal{R}=\mathcal{R}^*_{\eqref{extranonsense3}}\subset\mathcal{X}$ by \cite[Theorem 2.5]{Bam20c}.  \cite[Theorem 9.31(a)(c)]{Bam20b} and \eqref{extranonsense05} imply that
\begin{align*}
    \lim_{k\to\infty}u^{i_k}_{j_k,t_0}(y^{i_k}_{j_k})=u_{t_0}(y)=c_0,
\end{align*}
which contradicts \eqref{extranonsense4}; this finishes the proof.
\end{proof}

\subsection{Conclusion}

In summary, applying Proposition \ref{local-regularity} and Proposition \ref{local-noncollapsing} successively, for each $i$ large enough, we may pass each $\{\big(\mathcal{M}^{i}_{j},\nu_{x^{i}_{j};t}\big)\}_{j=1}^\infty$ to a subsequence, such that $\mathcal{N}_{y^i_j}(r^2)\ge-\varepsilon$ and $ u^i_{j,t_0}(y^i_j)>\tfrac{1}{2}c_0$ at the same time. By our assumption on $\varepsilon$, \cite[Theorem 10.3]{Bam20a}, and \cite[Theorem 6.1]{Bam20a}, we have that whenever $i$ is large enough, 
\begin{gather*}
    |{\Rm^i_j}|\le\tfrac{1}{2}r^{-2}\quad\text{on}\quad P(y^i_j;\tfrac{1}{2}r)\qquad \text{ for all } j\in\mathbb{N},
    \\
    |B(y^i_j,\tfrac{1}{2}r)|\ge c(n)(\tfrac{1}{2}r)^n\qquad \text{ for all } j\in\mathbb{N},
    \\
    u^i_{j,t_0}(y^i_j)>\tfrac{1}{2}c_0\qquad \text{ for all } j\in\mathbb{N}.
\end{gather*}

For each fixed $i$ large enough, we apply \cite[Lemma 9.33]{Bam20b} to the sequences $\{\big(\mathcal{M}^{i}_{j},\nu_{x^{i}_{j};t}\big)\}_{j=1}^\infty$ and $\{y^i_j\}_{j=1}^\infty$. Thus, there are positive constant $r^*=r^*(r,c_0)>0$ and $C=C(r,c_0)<+\infty$, independent of $i$, such that
\begin{enumerate}
    \item $\displaystyle P(y^i;r^*)\subset \mathcal{R}^i$ is unscathed,
    \item $\displaystyle |B(y^i,r^*)|\ge 1/C$,
    \item $\displaystyle |{\Rm^i}|\le C$ on $\displaystyle P(y^i;r^*)$,
\end{enumerate}
where the first conclusion is a direct consequence of \cite[Lemma 9.33]{Bam20b} and the last two are due to the definition of smooth convergence. It follows immediately from \cite[Definition 9.30]{Bam20b} that $y\in\mathcal{X}$ is a point of smooth convergence in \eqref{convergence-2}. This finishes the proof of the main theorem.

\bibliographystyle{amsalpha}

\newcommand{\alphalchar}[1]{$^{#1}$}
\providecommand{\bysame}{\leavevmode\hbox to3em{\hrulefill}\thinspace}
\providecommand{\MR}{\relax\ifhmode\unskip\space\fi MR }
\providecommand{\MRhref}[2]{%
  \href{http://www.ams.org/mathscinet-getitem?mr=#1}{#2}
}
\providecommand{\href}[2]{#2}

\noindent Mathematics Institute, Zeeman Building, University of Warwick, Coventry CV4 7AL, UK
\\ E-mail address: \verb"pak-yeung.chan@warwick.ac.uk"
\\

\noindent Department of Mathematics, Rutgers University, Piscataway, NJ 08854, USA
\\ E-mail address: \verb"zilu.ma@rutgers.edu"
\\

\noindent School of Mathematical Sciences, Shanghai Jiao Tong University, Shanghai, China, 200240
\\ E-mail address: \verb"sunzhang91@sjtu.edu.cn"

\end{document}